\documentclass[11pt]{amsart}

\usepackage{amssymb}
\usepackage{amsmath}
\usepackage{amsthm}
\usepackage{amscd}
\usepackage{amsfonts}
\usepackage{ascmac}
\usepackage[usenames]{color}
\usepackage{graphicx}
\usepackage{epstopdf}

\theoremstyle{plain}
\newtheorem{thm}{Theorem}[section]
\newtheorem{prop}[thm]{Proposition}
\newtheorem{lem}[thm]{Lemma}
\newtheorem{cor}[thm]{Corollary}
\newtheorem{defn}[thm]{Definition}

\newtheorem{rem}[thm]{Remark}

\def\vol{\mathop{\mathrm{Vol}}\nolimits}

\newcommand{\bfC}{{\mathbf C}}
\newcommand{\bfP}{{\mathbf P}}
\newcommand{\bfR}{{\mathbf R}}

\newcommand{\mapright}[1]{\smash{\mathop{   \hbox to 0.7cm{\rightarrowfill}}
  \limits^{#1}}}

\def\om{\omega}

\newcommand{\Fut}{\mathrm{Fut}}

\renewcommand{\emph}[1]{{\color{red} \it #1}}

\definecolor{orange}{cmyk}{0, 0.7, 1, 0}
\definecolor{light-green}{cmyk}{0.5, 0, 0.5, 0}
\definecolor{light-blue}{cmyk}{0.5, 0, 0, 0}
\definecolor{light-yellow}{cmyk}{0,0,0.6,0}
\definecolor{dark-green}{cmyk}{0.7, 0, 0.7, 0.5}

\title{On the existence problem of Einstein-Maxwell K\"ahler metrics}
\author{Akito Futaki and Hajime Ono}
\address{Graduate School of Mathematical Sciences, The University of Tokyo, 3-8-1 Komaba Meguro-ku Tokyo 153-8914, Japan}
\email{afutaki@ms.u-tokyo.ac.jp}
\address{Department of Mathematics, Saitama University, 255 Shimo-Okubo, Sakura-Ku,
Saitama 380-8570, Japan}
\email{hono@rimath.saitama-u.ac.jp}

\date{March 17, 2018}

\begin{document}
\begin{abstract} In this expository paper we review on the existence problem of 
Einstein-Maxwell K\"ahler metrics, and make several remarks.
Firstly, we consider a slightly more general set-up than Einstein-Maxwell K\"ahler metrics, 
and give extensions of volume minimization principle, the notion of toric K-stability and
other related results to the general set-up.
Secondly, we consider the toric case when the manifold is the one point blow-up
of the complex project plane and the K\"ahler class $\Omega$ is chosen so that the area of the exceptional curve is sufficiently close
to the area of the rational curve of self-intersection number 1. We observe by numerical analysis that there should be a Killing vector field $K$ which gives a toric K-stable pair 
$(\Omega, K)$ in the sense of Apostolov-Maschler.
\end{abstract}

\maketitle

\section{Introduction}\label{section1}

Let $(M,J)$ be a compact K\"ahler manifold of complex dimension $m$. A Hermitian metric
$\Tilde{g}$ of constant scalar curvature on $(M,J)$ is said to be 
a conformally K\"ahler, Einstein-Maxwell (cKEM for short) metric
 if there exists a positive smooth function $f$ on $M$  such that 
$g=f^2\Tilde{g}$  is K\"ahler and that the Hamiltonian vector field $K=J\mathrm{grad}_gf$ of $f$ with respect to the K\"ahler form $\omega_g$ of $g$ is a Killing vector field for both $g$ and $\Tilde{g}$.
In this case we call the K\"ahler metric $g$ 
an Einstein-Maxwell K\"ahler (EMK for short) metric. 
Let $\omega_0$ be a K\"ahler form,
and consider $\Omega = [\omega_0] \in H^2_{\mathrm{DR}}(M,\mathbf R)$ as a fixed K\"ahler class. 
We look for an
Einstein-Maxwell K\"ahler metric $g$ such that the K\"ahler form $\omega_g$ belongs to $\Omega$.

Let $G$ be a maximal torus of the reduced automorphism group, and pick $K \in \mathfrak g := \mathrm{Lie}(G)$.
Then the problem is to
find a $G$-invariant K\"ahler metric $g$ with its K\"ahler form $\omega_g \in \Omega$ such that

(i) $\Tilde g = f^{-2} g$ is a cKEM metric,

(ii) $J\mathrm{grad}_{g} f = K$.

The scalar curvature $s_{\Tilde{g}}$ of $\Tilde{g} = f^{-2}g$ is given by
\begin{equation}\label{scal1}
s_{\Tilde{g}} = f^{2} s_g -2(2m-1)f\Delta_g f - 2m(2m-1)|df|_g^2
\end{equation}
where $s_g$ is the scalar curvature of $g$ and $\Delta_g$ is the Hodge Laplacian with respect to $g$. 

Now, starting with a K\"ahler metric $g$ 
and a Killing potential $f$, for any real number $n \in \mathbf R$ with $n \ne 0,\ 1,\ 2$ 
and $k \in \mathbf R$ with $k \ne 0$ we define the 
$(g,f,k,n)$-scalar curvature $s_{g,f,k,n}$ by
\begin{equation}\label{scal2}
s_{g,f,k,n} = f^{-k}\left\{ s_g +k(n-1)\frac{1}{f}\Delta_g f +\frac{k}{4} (n-1)(4+2k-kn)\frac{1}{f^2}|df|_g^2\right\}.
\end{equation}
The case $n=2m$ is the scalar curvature $s_{\Tilde{g}}$ of the conformal metric $\Tilde{g}=f^k g$,
and for other values of $n$ such a meaning is lost. However, the cases of general
values of $n$ appear in natural contexts such as in \cite{ACGL17} and \cite{LejmiUpmeier17}.
Moreover, Lahdili proves in \cite{Lahdili17} and \cite{Lahdili17_2} results for cKEM metrics generalizing
to constant $(g,f,-2,n)$-scalar curvature.

In this expository paper we give extensions of the volume minimization principle \cite{FO17}, \cite{ICCMNotices17}, the notion of toric K-stability \cite{AM} for $k = -2$ and
other related results for cKEM metrics to the general set-up of constant $(g,f,k,n)$-scalar curvature.
We consider the toric case where the manifold is the one point blow-up
of the complex project plane and the K\"ahler class $\Omega$ is chosen so that the area of the exceptional curve is sufficiently close
to the area of the rational curve of self-intersection number 1. We observe by numerical analysis that there should be a Killing 
vector field $K$ which gives a toric K-stable pair 
$(\Omega, K)$ in the sense of Apostolov-Maschler. For this purpose we show in 
Theorem \ref{gen-D,W-Z} that we have only to consider the 
simple test configurations to test toric K-stability, extending the earlier works of 
Donaldson \cite{donaldson02}, Wang and Zhou \cite{Wang-Zhou11}, \cite{Wang-Zhou14}.

The rest of this paper is organized as follows. In section 2 we extend the volume minimization for 
Einstein-Maxwell K\"ahler metrics, see Theorem \ref{vol minimization1}. In section 3 we review the normalized Einstein-Hilbert functional, and study its relation to the volume functional and the Futaki invariant. In section 4 we consider the 
normalized Einstein-Hilbert functional on toric K\"ahler manifolds. In section 5 we review toric K-stability, and 
prove Theorem \ref{gen-D,W-Z}. We then review the result of our paper \cite{FO17} on the one-point blow-up
of $\bfC\bfP^2$ and show the graphics of the results of the numerical analysis which indicate that this case 
should be K-stable and there should be a conformally K\"ahler, Einstein-Maxwell metric.

\section{Volume minimization for Einstein-Maxwell K\"ahler metrics}

In this section we review the results in \cite{FO17} and extend them to constant $(g,f,k,n)$-scalar curvature.
Let $M$ be a compact smooth manifold. We denote by 
$\mathrm{Riem}(M)$ the set of all Riemannian metrics on $M$, 
by $s_g$ the scalar curvature of $g$, and
by $dv_g$ the volume form of $g$.
For any given positive smooth function $f$ and real numbers $n \in \mathbf R$ with $n \ne 0,\ 1,\ 2$
and $k \in \mathbf R$ with $k \ne 0$, 
we define $s_{g,f,k,n}$ by the same formula as \eqref{scal2}.
We put
\begin{equation}\label{scal3}
S(g,f,k,n) := \int_M s_{g,f,k,n}\,f^{\frac{nk}2}dv_g
\end{equation}
and call it the total $(g,f,k,n)$-scalar curvature, and put
\begin{equation}\label{scal4}
\vol(g,f,k,n) := \int_M\,f^{\frac{nk}2}dv_g
\end{equation}
and call it the $(g,f,k,n)$-volume. 

Let $f_t$ be a smooth family of positive
functions such that $f_0=f,d/{dt}\vert_{t=0}f_t=\phi$. Then by straightforward computations we have
\begin{equation}\label{eq:2.2}
\left.\dfrac{d}{dt}\right\vert_{{t=0}}S(g,f_t,k,n)=\frac{k}2 (n-2)\int_M\,s_{g,f,k,n}\,\phi\,
f^{\frac{nk}2 - 1}\,dv_g
\end{equation}
and
\begin{equation}\label{eq:2.3}
\left.\dfrac{d}{dt}\right\vert_{{t=0}}\vol(g,f_t,k,n)=\frac{nk}2\int_M\, \phi\, f^{\frac{nk}2 - 1}\,dv_g.
\end{equation}

Now we consider a compact K\"ahler manifold $(M,J)$ of complex dimension $m$. 
As in section \ref{section1}, let $G$ be a maximal torus of the reduced automorphism group, and take $K \in \mathfrak g := \mathrm{Lie}(G)$.
Consider a fixed K\"ahler class
$\Omega$ on $(M,J)$, and denote by
$\mathcal K_\Omega^G$ the space of $G$-invariant K\"ahler metrics $\omega$
in $\Omega$. For any $(K,a,g)\in \mathfrak g\times \mathbf{R}\times
\mathcal K_\Omega^G$, there exists a unique function $f_{K,a,g}
\in C^\infty(M,\mathbf{R})$
satisfying the following two conditions:
\begin{equation}\label{eq:2.4}
\iota_K\omega=-df_{K,a,g},\ \ 
\int_Mf_{K,a,g}\frac{\omega^m}{m!}=a.
\end{equation}
By \eqref{eq:2.4}, 
it is easy to see that $f_{K,a,g}$ has the following properties:
\begin{equation}\label{eq:2.5}
f_{K+H,a+b,g}=f_{K,a,g}+f_{H,b,g}
\end{equation}
\begin{equation}\label{eq:2.6}
f_{0,a,g}=\dfrac{a}{\vol(M,\omega)}
\end{equation}
\begin{equation}\label{eq:2.7}
f_{CK,Ca,g}=Cf_{K,a,g}
\end{equation}

Hereafter the K\"ahler metric $g$ and its K\"ahler form $\omega_g$ are often identified, and $\omega_g$ is
often denoted by $\omega$.
Noting that $\min\{f_{K,a,g}\,|\,x \in M\}$ is independent of $g\in \mathcal K_\Omega^G$ (this follows from the convexity of
moment map images and the fact that the vertices do not move even if we change the K\"ahler metric in the fixed K\"ahler class $\Omega$), 
we put
\begin{eqnarray}
\mathcal P_\Omega^G&:=&\{(K,a)\in \mathfrak g\times \mathbf{R}\,|\, f_{K,a,g} > 0\}. \label{eq:2.8}
\end{eqnarray}
Note that the right hand side of \eqref{eq:2.8} is independent of $g\in \mathcal K_\Omega^G$ again since the moment polytope
is independent of $g\in \mathcal K_\Omega^G$. 
Fixing $(K,a) \in P^G_{\Omega}$, $n \in \mathbf R$ and $k \in \mathbf R$, put
\begin{equation}\label{eq:2.9}
c_{\Omega,K,a,k,n}:=\dfrac
{\displaystyle{\int_M\,s_{g,f_{K,a,g},k,n}\,f_{K,a,g}^{\frac{kn}2-1}\,
\frac{\omega^m}{m!}}}
{\displaystyle{\int_M\,f_{K,a,g}^{\frac{kn}2-1}\,
\frac{\omega^m}{m!}}}
\end{equation}
and
\begin{equation}\label{eq:2.10}
d_{\Omega,K,a,k,n}:=
\dfrac{S(g, f_{K,a,g},k,n)}{\vol(g, f_{K,a,g},k,n)}
=\dfrac
{\displaystyle{\int_M\,s_{g,f_{K,a,g},k,n}\,f_{K,a,g}^{\frac{kn}2}\,
\frac{\omega^m}{m!}}}
{\displaystyle{\int_M\,f_{K,a,g}^{\frac{kn}2}\,
\frac{\omega^m}{m!}}}
\end{equation}
Then $c_{\Omega,K,a,k,n}$ and $d_{\Omega,K,a,k,n}$ are constants independent of the choice of $g\in \mathcal K_\Omega^G$
since the integrands of (\ref{eq:2.9}) and (\ref{eq:2.10}) are part of equivariant cohomology, see e.g. \cite{futakimorita85}, \cite{futaki88}, 
\cite{futakimabuchi02}.
Since $P^G_{\Omega}$ is a cone in $\mathfrak g\times \mathbf{R}$ by \eqref{eq:2.7}, with $n$ and $k$ fixed we consider its slice
\begin{equation}\label{eq:2.14}
\Tilde{\mathcal P}_{\Omega}^G:=\left\{(K,a)\in \mathcal P_\Omega^G\,\Big|\,
d_{\Omega,K,a,k,n}=\gamma \right\}
\end{equation}
where $\gamma$ is chosen to be $-1$, $0$ or $1$ depending on the sign of $d_{\Omega,K,a,k,n}$.
Let $(K(t),a(t)),\ t\in (-\varepsilon,
\varepsilon)$ be a smooth curve in $\Tilde{\mathcal P}_{\Omega}^G$
such that $(K(0),a(0))=(K,a),(K'(0),a'(0))=(H,b)$.
Then 
$$
S(g, f_{K(t),a(t),g},k,n)= \gamma \vol(g, f_{K(t),a(t),g},k,n)
$$
holds for any $t\in (-\varepsilon,
\varepsilon)$. By differentiating this equation at $t=0$ and noting $k \ne 0$, we have
\begin{equation}\label{eq:2.15}
(n-2)\int_M
\,s_{g, f_{K,a,g},k,n}f_{H,b,g}\,f_{K,a,g}^{\frac{nk}2 - 1}
\dfrac{\omega^m}{m!}=n\gamma
\int_M
\,f_{H,b,g}\,f_{K,a,g}^{\frac{nk}2 - 1}
\dfrac{\omega^m}{m!}.
\end{equation}
The linear function $\Fut^G_{\Omega,K,a,k,n}:\mathfrak g\to \bfR$ defined by
\begin{equation}\label{eq:2.16}
\Fut^G_{\Omega,K,a,k,n}(H):=
\int_M
(s_{g,K,a,k,n}-c_{\Omega,K,a,k,n})\,f_{H,b,g}\,f_{K,a,g}^{\frac{nk}2 - 1}\,
\dfrac{\om_g^m}{m!}
\end{equation}
is independent of the choice of K\"ahler metric $g\in \mathcal K^G_\Omega$
and $b \in \bfR$ (\cite{AM}). 
If there exists a K\"ahler metric $g \in \mathcal K^G_\Omega$
such that $\Tilde{g} = f^{k}_{K,a,g} g$ is a constant $(g,f,k,n)$-scalar curvature metric, then
$\Fut^G_{\Omega,K,a,k,n}$ vanishes identically.

For the path $(K(t),a(t)),\ t\in (-\varepsilon,
\varepsilon)$ in $\Tilde{\mathcal P}_\Omega^G$
with $(K(0),a(0))=(K,a)$, $(K'(0),a'(0))=(H,b)$ we have from \eqref{eq:2.15}
\begin{equation}\label{eq:2.17}
\begin{split}
\Fut_{\Omega,K,a,k,n}^G(H)&=\left(
\dfrac{n\gamma}{n-2}-c_{\Omega,K,a,k,n}\right)\int_M\,f_{H,b,g}\,
f_{K,a,g}^{\frac{nk}2 -1}\,\dfrac{\omega^m}{m!}\\
&=\left(\dfrac{n\gamma}{n-2}-c_{\Omega,K,a,k,n}\right)\frac{2}{nk}
\left.\dfrac{d}{dt}\right|_{t=0}\vol(g, f_{K(t),a(t),g},k,n).
\end{split}
\end{equation}
If there exists a constant $(g,f,k,n)$-scalar curvature metric
$\Tilde{g} = f^{k}_{K,a,g} g$ with $g\in \mathcal K^G_\Omega$,
then 
$$c_{\Omega,K,a,k,n}=d_{\Omega,K,a,k,n}=\gamma$$ 
and 
$$\Fut_{\Omega,K,a,k,n}^G(H)=0.$$
Therefore for $\gamma = \pm1$ we have
\begin{equation}\label{eq:2.18}
\left.\dfrac{d}{dt}\right|_{t=0}\vol(g, f_{K(t),a(t),g},k,n)=0.
\end{equation}
The case of $\gamma = 0$ can be treated separately, see \cite{FO17}.

We summarize the result as follows.

\begin{thm}\label{vol minimization1}
Let $\Omega$ be a fixed K\"ahler class, and $n \ne 0,\ 1,\ 2$ and $k \ne 0$ be fixed real numbers.
Suppose that the pair $(K,a)$ of Killing vector field $K$ and normalization constant $a$ belongs to
$\Tilde{\mathcal P}_{\Omega}^G$. If there exists a $G$-invariant K\"ahler metric $g$ in 
the K\"ahler class $\Omega$, i.e. $g \in \mathcal K^G_\Omega$,  such that the 
$(g,f,k,n)$-scalar curvature is constant for the Killing Hamiltonian function $f =  f_{K,a,g}$
then $(K,a)$ 
is a critical point of 
$\vol_{n,k}:
\Tilde{\mathcal P}_{\Omega}^G\to \mathbf{R}$
given by 
\begin{eqnarray*}
\vol_{n,k}(K,a) &:=&\vol(g, f_{K,a,g},k,n)\\
&=& 
\int_M\,f_{K,a,g}^{\frac{nk}2}dv_g
\end{eqnarray*}
for $(K,a)\in 
\Tilde{\mathcal P}_{\Omega}^G.$
Further, $(K,a)$ 
is a critical point of 
$\vol_{n,k}:
\Tilde{\mathcal P}_{\Omega}^G\to \mathbf{R}$
if and only if 
$\Fut_{\Omega,K,a,k,n}^G\equiv 0$.
\end{thm}

\begin{cor}\label{vol minimization2}
Let $\Omega$ be a fixed K\"ahler class.
Take $n=2m$ and $k=-2$, and 
let $(K,a)\in \Tilde{\mathcal P}_{\Omega}^G$. If there exists a conformally K\"ahler,
Einstein-Maxwell metric $\Tilde{g} = f^{-2}_{K,a,g} g$ with $g\in \mathcal K^G_\Omega$, 
then $(K,a)$ 
is a critical point of 
$\vol:
\Tilde{\mathcal P}_{\Omega}^G\to \mathbf{R}$
given by $\vol(K,a):=\vol(g, f_{K,a,g},-2,2m)$ for $(K,a)\in 
\Tilde{\mathcal P}_{\Omega}^G$.
Further, $(K,a)$ 
is a critical point of 
$\vol:
\Tilde{\mathcal P}_{\Omega}^G\to \mathbf{R}$
if and only if 
$\Fut_{\Omega,K,a,-2,2m}^G\equiv 0$.
\end{cor}

For a given K\"ahler class $\Omega$ the critical points of $\vol:
\Tilde{\mathcal P}_{\Omega}^G\to \mathbf{R}$ are not unique in general
as can be seen from LeBrun's construction \cite{L2}.

\section{The normalized Einstein-Hilbert functional}
In the previous section we confined ourselves to the view point from the volume functional.
In the present section we see that, when restricted to $\Tilde{\mathcal P}_{\Omega}^G$, 
considering the volume functional is essentially the same as considering the normalized Einstein-Hilbert functional.
The normalized Einstein-Hilbert functional $EH : \mathrm{Riem}\,(M)\to \mathbf{R}$ on an
$n$-dimensional compact Riemannian
manifold is the functional on $\mathrm{Riem}(M)$
defined by 
$$EH(g):=\dfrac{S(g)}{(\vol(g))^{\frac{n-2}{n}}}$$
where $S(g)$ and $\vol(g)$ are respectively the total scalar curvature and the volume of $g$.
It is a standard fact that the critical points of $EH$ are Einstein metrics, and that, when restricted 
to a conformal class, the critical points are metrics of constant scalar curvature.

Let us see this in a slightly different setting.
In the equation \eqref{scal2}, let us replace $s_g$ by a smooth function $\varphi$, and put
\begin{equation}\label{scal12}
s_{g,f,k,n,\varphi} = f^{-k}\left\{ \varphi +k(n-1)\frac{1}{f}\Delta_g f +\frac{k}{4} (n-1)(4+2k-kn)\frac{1}{f^2}|df|_g^2\right\}.
\end{equation}
Accordingly, we may replace \eqref{scal3} by 
\begin{equation}\label{scal13}
S(g,f,k,n,\varphi) := \int_M s_{g,f,k,n,\varphi}\,f^{\frac{nk}2}dv_g,
\end{equation}
and replace the normalized Einstein-Hilbert functional by
$$EH(g,f,k,n,\varphi):=\dfrac{S(g,f,k,n,\varphi)}{(\vol(g,f,k,n,\varphi))^{\frac{n-2}{n}}}.$$
As before, let $f_t$ be a smooth family of positive
functions such that $f_0=f,d/{dt}\vert_{t=0}f_t=\phi$. 
Then one can show
\begin{eqnarray}\label{scal15}
&&\left.\frac{d}{dt} \right|_{t=0} EH(g,f_t,k,n,\varphi) \\
&& \qquad =
\frac{(n-2)k}2 \vol(g,f,k,n,\varphi)^\frac{2-n}{n} \nonumber\\
&&\qquad\qquad \cdot \left\{\int_M \left(s_{g,f,k,n,\varphi} - \dfrac{S(g,f,k,n,\varphi)}{\vol(g,f,k,n)}\right)\phi\,f^{\frac{nk}2 -1} dv_g\right\}.\nonumber
\end{eqnarray}
Thus we have shown
\begin{prop}\label{EH1} The function $s_{g,f,k,n,\varphi}$ satisfies
$$ s_{g,f,k,n,\varphi} = \mathrm{constant} $$
if and only if $f$ is a critical point of the functional 
$f \mapsto EH(g,f,k,n,\varphi)$.
\end{prop}


Let us return to the situation of the previous section where we considered a 
compact K\"ahler manifold with a maximal torus $G$ of the reduced automorphisms group,
with a fixed K\"ahler class $\Omega$. 
Taking $\varphi$ to be the $(g,f,k,n)$-scalar curvature, we consider the Einstein-Hilbert
functional $EH(g,f,k,n):= EH(g,f,k,n,s_{g,f,k,n})$.
By the same reasoning from equivariant cohomology again, for a fixed $(K,a)$,
$EH(g,f_{K,a,g},k,n)$ is independent of the choice of $g\in \mathcal K^G_\Omega$.
Set $EH_{k,n}(K,a):=EH(g,f_{K,a,g},k,n)$. Then using \eqref{eq:2.5}, \eqref{eq:2.6} and \eqref{eq:2.7}, we see
\begin{eqnarray}\label{eq:3.3}
&&\left.\dfrac{d}{dt}\right\vert_{{t=0}}EH_{k,n}(K+tH,a)\\
&&=
\dfrac{(n-2)k}{2\vol_n(K,a)^{\frac{n-2}{n}}}\int_M\left(s_{g,f_{K,a,g},k,n}-d_{\Omega,K,a,k,n}\right)
f_{K,a,g}
^{\frac{nk}2-1}f_{H,0,g}\frac{\om_g^m}{m!}\nonumber
\end{eqnarray}
and
\begin{eqnarray}\label{eq:3.4}
&&\left.\dfrac{d}{dt}\right\vert_{{t=0}}EH_{k,n}(K,a+tb)\\
&&=\dfrac{(n-2)kb}{2\vol_n(K,a)^{\frac{n-2}{n} + 1}}
(c_{\Omega,K,a,n}-d_{\Omega,K,a,n})\int_M
f_{K,a,g}^{\frac{nk}2-1}
\frac{\om_g^m}{m!}.\nonumber
\end{eqnarray}
If there exist $g \in \mathcal K^G_\Omega$, $K$ and $a$ such that $s_{g, f_{K,a,g},k,n}$ is constant, then
\begin{equation}\label{critical} 
s_{g, f_{K,a,g},k,n} =  c_{\Omega,K,a,k,n} = d_{\Omega,K,a,k,n}, 
\end{equation}
and thus 
the pair $(K,a)$
is a critical point of the function $EH_{k,n} : \mathcal P_\Omega^G \to \mathbf R$ given by
\begin{equation}\label{eq:3.5}
(K,a)\mapsto EH_{k,n}(K,a):=EH(g,f_{K,a,g},k,n).
\end{equation}
Conversely, suppose that $(K,a)$ is a critical point of $EH_{k,n} : \mathcal P_\Omega^G \to \mathbf R$.
Then one can see $(K,a)$ satisfies
$c_{\Omega,K,a,k,n} = d_{\Omega,K,a,k,n}$.
Hence, by \eqref{eq:2.16} and \eqref{eq:3.3}, 
$\mathrm{Fut}^G_{\Omega,K,a,k,n}$ vanishes.
More direct relation between the volume functional and the Einstein-Hilbert functional
can be seen as follows.
\begin{rem}\label{restriction}
Since $EH_{k,n}$ is homogeneous of degree $0$ on $\mathcal P_\Omega^G$ we may restrict $EH_{k,n}$ to the slice
\begin{equation}\label{eq:3.8}
\tilde{\mathcal P}^G_{\Omega,n}:=
\{(K,a)\in \mathcal P^G_{\Omega,n}\,|\,
d_{\Omega,k,a,k,n}=\gamma\}
\end{equation}
Then 
\begin{equation}\label{eq:3.9}
EH_{k,n}(K,a)=\gamma \vol_{k,n} (K,a)^{\frac{2}{n}}
\end{equation}
on 
$\tilde{\mathcal P}^G_{\Omega,n}$.
This shows that the volume minimization Theorem \ref{vol minimization1} is equivalent to
finding a critical point of the Einstein-Hilbert functional.
\end{rem}


\section{The
normalized Einstein-Hilbert functional for toric K\"ahler manifolds.}

In this section, we give the explicit formula for
the Futaki invariant and
the normalized Einstein-Hilbert functional
when $(M,J,\omega)$ is a compact toric K\"ahler manifold and $k=-2$.

Let $(M,\omega)$ be a $2m$-dimensional compact toric manifold
and $\mu:M\to \mathbf{R}^m$ the moment map.
It is well-known that the image of $\mu$,
$\Delta:=\mathrm{Image}\,\mu$, is an $m$-dimensional Delzant polytope
in $\mathbf{R}^m$. A $T^m$-invariant, $\omega$-compatible complex
structure $J$ on $M$ gives
a convex function $u$, called a symplectic potential, on $\Delta$ as follows.
For the action-angle coordinates $(\mu_1,\dots,\mu_m,
\theta_1,\dots,\theta_m)
\in \Delta\times T^m$, there exists a smooth
convex function $u$ on $\Delta$ which satisfies
$$
J\frac{\partial }{\partial \mu_i}=
\sum _{j=1}^mu_{,ij}\frac{\partial }{\partial \theta_j},\ \ 
J\frac{\partial }{\partial \theta_i}=\sum_{j=1}^m
\mathbf H^u_{ij}\frac{\partial }{\partial \mu_j},
$$
where, for a smooth function $\varphi $ of $\mu=(\mu_1,\dots,\mu_m)$,
we denote by $\varphi_{,i}$ the partial derivative $\partial \varphi/\partial \mu_i$
and by $\mathbf H^u=(\mathbf H^u_{ij})$ the inverse matrix of the Hessian $(u_{,ij})$ of $u$.
Conversely, if we give a smooth convex function $u$ on $\Delta$
satisfying some boundary condition,
by the formula above, we can recover a $T^m$-invariant $\omega$-compatible
complex structure on $M$, see \cite{Abreu} for more detail.

Let $u$ be a symplectic potential on $\Delta$.
Then the toric K\"ahler metric $g_J=\omega(\cdot,J\cdot)$ is
represented as
\begin{equation}\label{toric metric}
g_J=\sum_{i,j=1}^mu_{,ij}d\mu_i d\mu_j+\sum_{i,j=1}^m
\mathbf H^u_{ij}d\theta_i d\theta_j.
\end{equation}
According to Abreu \cite{Abreu}, the scalar curvature $s_J$ of  $g_J$
is
\begin{equation}\label{scalar curv}
s_J=
-\sum_{i,j=1}^m\mathbf H^u_{ij,ij}.
\end{equation}

In this case, a Killing potential is an affine linear function positive on $\Delta$.
Fix a Killing potential $f$. Then $(g_J,f,k,n)$-scalar curvature
$s_{J,f,k,n}$ is given by

\begin{equation}\label{t-scal1}
s_{J,f,k,n}=f^{-k}s_J+\frac{4(n-1)}{n-2}f^{-\frac{k(n+2)}{4}}\Delta_Jf^{\frac{k(n-2)}{4}},
\end{equation}
where $\Delta_J=\Delta_{g_J}$.
For a smooth function $\varphi$ of $\mu_1,\dots,\mu_m$, 
$$
\Delta_J\varphi=-\sum_{i,j=1}^m \{\varphi_{,ij}\mathbf H^u_{ij}+\varphi_{,i}\mathbf H^u
_{ij,j}\}
$$
holds (see the equation $(20)$ in \cite{AM}).
Since $f$ is affine linear, we have
\begin{equation}\label{laplacian}
\begin{split}
&\Delta_Jf^{\frac{k(n-2)}{4}}\\
&=
-\frac{k(n-2)}{4}f^{\frac{k(n-2)}{4}}\sum_{i,j=1}^m
\left\{
\left(
\frac{k(n-2)}{4}-1
\right)\frac{f_{,i}f_{,j}}{f^2}\mathbf H^u_{ij}+\frac{f_{,i}}{f}\mathbf H^u_{ij,j}
\right\}.
\end{split}
\end{equation}

By \eqref{scalar curv}, \eqref{t-scal1} and \eqref{laplacian},
the $(g_J,f,k,n)$-scalar curvature is

\begin{equation}\label{t-scal2}
\begin{split}
&s_{J,f,k,n}\\
&=
-f^{-k}\sum_{i,j=1}^m \left\{\mathbf H^u_{ij,ij}+
\frac{k(n-1)}{f}
f_{,i}\mathbf H^u_{ij,j}+
\frac{k(n-1)}{f^2}\left(\frac{k(n-2)}{4}-1\right) f_{,i} f_{,j}
\mathbf H^u_{ij}
\right\}.
\end{split}
\end{equation}

On the other hand, for any $\alpha\in \mathbf R$,
\begin{equation}\label{d-Hess1}
\sum_{i,j=1}^m
\left(
f^\alpha \mathbf H^u_{ij}
\right)_{,ij}=f^\alpha \sum_{i,j=1}^m
\left\{
\mathbf H^u_{ij,ij}+\frac{2\alpha}{f}f_{,i}\mathbf H^u_{ij,j}+
\frac{\alpha(\alpha-1)}{f^2}f_{,i}f_{,j}\mathbf H^u_{ij}
\right\}
\end{equation}
holds. We easily see that
$2\alpha=k(n-1)$ and $\alpha(\alpha-1)=k(n-1)(k(n-2)/4-1)$ hold
if and only if $k=-2$ and $\alpha=1-n$. In this case, we have
\begin{equation}\label{t-scal3}
s_{J,f,-2,n}f^{-1-n}=-\sum_{i,j=1}^m
\left(
f^{1-n}\mathbf H^u_{ij}
\right)_{,ij}.
\end{equation}
By Lemma $2$ in \cite{AM}, for any smooth function $\phi$
on $\mathbf{R}^m$,
\begin{equation}\label{int-by-parts1}
\int_\Delta 
\phi
\sum_{i,j=1}^m
\left(
f^{1-n}\mathbf H^u_{ij}
\right)_{,ij}
\, d\mu
=
\int_\Delta
 f^{1-n}
\sum_{i,j=1}^m
\mathbf H^u_{ij}\phi_{,ij}
\, d\mu
-2
\int_{\partial \Delta}
f^{1-n}\phi \, d\sigma.
\end{equation}
In particular, when $\phi$ is an affine function
\begin{equation}\label{int-by-parts2}
\int_\Delta 
\phi
\sum_{i,j=1}^m
\left(
f^{1-n}\mathbf H^u_{ij}
\right)_{,ij}
\, d\mu
=
-2
\int_{\partial \Delta}
f^{1-n}\phi \, d\sigma
\end{equation}
holds. 
Hence, if we define the constant $c_{\Delta,f,-2,n}$
as
$$
c_{\Delta,f,-2,n}=2\,\dfrac{\displaystyle{\int_{\partial \Delta}f^{1-n}\,d\sigma}}
{\displaystyle{\int_\Delta f^{-1-n}\,d\mu}},
$$
the Futaki invariant \eqref{eq:2.16} is given by
\begin{equation}\label{t-F-inv}
\mathrm{Fut}_{\Delta,f,-2,n}(\phi)
=2\int_{\partial \Delta}f^{1-n}\phi\,d\sigma
-c_{\Delta,f,-2,n}
\int_\Delta
f^{-1-n}\phi \, d\mu
\end{equation}
for any linear function $\phi$ on $\mathbf{R}^m$.

By \eqref{t-scal3} and \eqref{int-by-parts2},
$EH(g_J,f,-2,n)$ is given by
\begin{equation}\label{toric-EH}
EH_{-2,n}(f):=
EH(g_J,f,-2,n)
=\text{Const.}
\dfrac{\displaystyle{\int_{\partial \Delta}f^{2-n}\,d\sigma }}
{\displaystyle{\left(
\int_\Delta f^{-n}\, d\mu
\right)^{\frac{n-2}{n}}}}.
\end{equation}

If there exists a symplectic potential $u$ such that
the $(g_J,f,-2,n)$-scalar curvature is constant,
then $\mathrm{Fut}_{\Delta,f,-2,n}$ vanishes identically and
$f$ is a critical point of $EH_{-2,n}$.

\section{Toric K-stability}

Let $(M,\omega)$ be a $2m$-dimensional
compact toric manifold with the moment image $\Delta\subset
\mathbf{R}^m$.
Following the argument by Donaldson in \cite{donaldson02},
we may define the Donaldson-Futaki invariant
with respect to a positive affine function $f$ on $\Delta$ 
as
\begin{equation}\label{DF inv}
\mathrm{DF}_{\Delta,f,n}(\phi)=
2\int_{\partial \Delta}f^{1-n}\phi\,d\sigma
-c_{\Delta,f,-2,n}
\int_\Delta
f^{-1-n}\phi \, d\mu
\end{equation}
for a convex function $\phi$ on $\Delta$, see also \cite{AM}.
For any affine fuction $\phi$, 
$$\mathrm{Fut}_{\Delta,f,-2,n}(\phi)=\mathrm{DF}_{\Delta,f,n}(\phi).$$
We can prove the following straightforward analogue of the results
in \cite{donaldson02}:

\begin{thm}\label{ex-to-semistable}
Suppose that there exists a symplectic potential $u$ on $\Delta$
such that the $(g_J,f,-2,n)$-scalar curvature is a constant $c$.
Then $c=c_{\Delta,f,-2,n}$ and
$\mathrm{DF}_{\Delta,f,n}(\phi)\ge 0$ for any smooth convex function
$\phi$ on $\Delta$. Equality holds if and only if $\phi$ is affine.
\end{thm}
\begin{proof}
Suppose that $s_{J,f,-2,n}=c$.
Then
$$
c\int_\Delta f^{-1-n}\,d\mu=
- \int_\Delta\,\sum_{i,j=1}^m(f^{1-n}\mathbf H^u_{ij})_{,ij}\,d\mu
=2\int_{\partial \Delta}f^{1-n}\,d\sigma
$$
by \eqref{t-scal3} and \eqref{int-by-parts2}.
Hence $c=c_{\Delta,f,-2,n}$. By \eqref{int-by-parts1},
\begin{equation}\label{DF inv2}
\begin{split}
\mathrm{DF}_{\Delta,f,n}(\phi) &=
-\int_\Delta \left(
c_{\Delta,f,-2,n}f^{-1-n}+\sum_{i,j=1}^m\left(
f^{1-n}\mathbf H^u_{ij}
\right)_{,ij}
\right)\phi\,d\mu\\
&\hspace{4mm}+\int_\Delta f^{1-n}\sum_{i,j=1}^m\mathbf H^u_{ij}\phi_{,ij}\,d\mu\\
&=\int_\Delta f^{1-n}\sum_{i,j=1}^m\mathbf H^u_{ij}\phi_{,ij}\,d\mu\ge 0.
\end{split}
\end{equation}
\end{proof}

\begin{defn}\label{K-stable}
Let $\Delta\subset \mathbf{R}^m$ be a Delzant polytope, $n\not=0,1,2$ and $f$ a
positive affine function on $\Delta$. $(\Delta,f,n)$ is K-semistable
if $\mathrm{DF}_{\Delta,f,n}(\phi)\ge 0$ for any piecewise linear convex function
$\phi$ on $\Delta$. $(\Delta,f,n)$ is K-polystable if
it is K-semistable and the equality $\mathrm{DF}_{\Delta,f,n}(\phi)=0$ is only
possible for $\phi$ affine linear.
\end{defn}

Since any piecewise linear convex function on $\Delta$ can be approximated by
smooth convex functions on $\Delta$, the existence of a constant
$(g_J,f,-2,n)$-scalar curvature metric implies the K-semistability
of $(\Delta,f,n)$.

We next consider compact toric surfaces and prove
that the positivity of Donaldson-Futaki invariant
for simple piecewise linear functions implies
K-polystability. This is a 
generalization of
the result by Donaldson \cite{donaldson02} and Wang-Zhou
\cite{Wang-Zhou11,Wang-Zhou14}.
The proof is similar to the one given in \cite{Wang-Zhou14}, but
to make this paper as self-contained as possible,
we give a proof here.

Let $P\subset \mathbf{R}^m$ be an $m$-dimensional open convex polytope,
$P^*$ a union of $P$ and the facets of $P$. Denote
$$
\mathcal C_1:=\{u:P^*\to \mathbf R\text{, convex}\,|\,
\int_{\partial P}u\,d\sigma<\infty\}.
$$
For positive bounded functions $\alpha,\beta$ on $\bar{P}$
and an affine function $A$ on $\mathbf{R}^m$,
we define the linear functional $\mathcal L$ on $\mathcal C_1$ as
\begin{equation}\label{gDF inv}
\mathcal L(u):=\int_{\partial P}\alpha u\,d\sigma
-\int_PA\beta u\,d\mu.
\end{equation}

\begin{thm}\label{gen-D,W-Z}
Suppose that $\mathcal L(f)=0$ for any affine function $f$ on $\mathbf{R}^m$.
When $m=2$, the following two conditions are equivalent.
\begin{enumerate}
\item $\mathcal L(u)\ge 0$ for any $u\in \mathcal C_1$ and
the equality holds if and only if $u$ is affine.
\item $\mathcal L(u)>0$ for any simple piecewise linear convex function $u$
with nonempty
crease.
\end{enumerate}
\end{thm}

Here a convex function $u$ is simple piecewise linear, sPL for short,
if $u=\max \{L,0\}$ for a non-zero affine function $L$.
The crease of sPL convex function $u$ is the intersection of 
$P$ and $\{L=0\}$.

\begin{proof}
It is sufficient to prove that (2) implies (1).
Suppose that
$\mathcal L$ is positive for any sPL convex function with nonempty crease. 
Moreover we assume
the case (1) does not occur, that is,
the one of the following holds:

\vspace{2mm}

$\circ$ There exists $v\in \mathcal C_1$ such that $\mathcal L(v)<0$.

$\circ$ For any $u\in \mathcal C_1$, $\mathcal L(u)\ge 0$ and
there exists $v\in \mathcal C_1\setminus\{\text{affine function}\}$\\
\hspace{6mm} such that $\mathcal L(v)=0$.

\vspace{2mm}

We fix $p_0\in P$ and denote
$$
\Tilde{\mathcal C}_1:=\left\{u\in \mathcal C_1\,|\,
\int_{\partial P}\alpha u \, d\sigma=1,
\inf_{P}u=u(p_0)=0\right\}.
$$
Since $\mathcal L$ vanishes on the set of affine functions
and $\mathcal L(cu)=c\mathcal L(u)$ for any $c>0$ and $u\in \mathcal {C}_1$,
we may assume $v$ in the condition above is an element of $\Tilde{\mathcal C}_1$.

\begin{lem}\label{bdd}
The functional $\mathcal L:\Tilde{\mathcal C}_1\to \mathbf R$
is bounded from below.
\end{lem}
\begin{proof}
By Lemma $5.1.3$ in \cite{donaldson02},
there exists a constant $C>0$ such that
$$
\int_P u\,d\mu \le C\int_{\partial P}u\,d\sigma
$$
for all $u\in \Tilde{C}_1$. Since $\alpha,\beta$ are positive and bounded
on $\bar{P}$
$$
\int_P\beta u\,d\mu \le \sup_{\bar{P}}\beta \int_Pu\,d\mu
\le \dfrac{C\sup_{\bar{P}} \beta}
{\inf_{\bar{P}} \alpha}=:C'
$$
for $u\in\Tilde{\mathcal C}_1$. Hence, on $\Tilde{\mathcal C}_1$,
$$
\mathcal L(u)
=1-\int_P A\beta u\,d\mu
\ge 1-\max_{\bar{P}}\lvert A\rvert \int_P\beta u\,d\mu
\ge 1-\max_{\bar{P}}\lvert A\rvert C'.
$$
\end{proof}
By assumption, $\inf_{\Tilde{C}_1}\mathcal L\le 0$.
Moreover we see that there exists $u_0\in \Tilde{C}_1$ which attains
the infimum of $\mathcal L$ on $\Tilde{\mathcal C}_1$ 
by the same argument with the proof of Lemma $4.2$ in \cite{Wang-Zhou14}
as follows. Let $\{u_k\}$ be a sequence in $\Tilde{C}_1$ with
$\lim_{k\to \infty}\mathcal L(u_k)=\inf_{\Tilde{C}_1}\mathcal L$.
By Lemma \ref{bdd} above and Corollary $5.2.5$ in \cite{donaldson02},
there is the limit function $u_0$ convex on $P^*$.
More precisely,
$$
u_0(p)=
\begin{cases}
\displaystyle{\lim_{k\to \infty}u_k(p)} & \text{if }p\in P\\
\displaystyle{\lim_{t\nearrow 1}u_0((1-t)p_0+tp) }& \text{if }p \text{ is in a facet of }P.
\end{cases}
$$
The limit function $u_0$ satisfies
$$
\int_PA\beta u_0\,d\mu=\lim_{k\to \infty}\int_PA\beta u_k\,d\mu
\text{ and }\inf_Pu_0=u_0(p_0)=0.
$$
By convexity, $\displaystyle{\int_{\partial P}\alpha u_0\,d\sigma\le 1}$.
Suppose that $\displaystyle{\int_{\partial P}\alpha u_0\,d\sigma< 1}$.
Then
\begin{align*}
\mathcal L(u_0) 
&=
\int_{\partial P}\alpha u_0\,d\sigma-\int_PA\beta u_0\,d\mu<1-\int_PA\beta u_0\,d\mu\\
&=\lim_{k\to \infty}\mathcal L(u_k)=\inf_{\Tilde{C}_1}\mathcal L\le 0.
\end{align*}
On the other hand, since $\Tilde{u}_0:=
\displaystyle{\left(\int_{\partial P}\alpha u_0\,d\sigma\right)^{-1}}u_0
\in \Tilde{\mathcal C}_1$,
$$
\left(\int_{\partial P}\alpha u_0\,d\sigma\right)^{-1}\mathcal L(u_0)
=\mathcal L(\Tilde{u}_0)\ge \inf_{\Tilde{\mathcal C}_1}\mathcal L.
$$
Hence $\mathcal L(u_0)<\displaystyle{\left(\int_{\partial P}\alpha u_0\,d\sigma\right)
^{-1}}\mathcal L(u_0)$. Since $\mathcal L(u_0)<0$, 
$\displaystyle{\int_{\partial P}\alpha u_0\,d\sigma> 1}$.
It is a contradiction.
Therefore $u_0\in \Tilde{\mathcal C}_1$ and it attains the infimum of $\mathcal L$
on $\Tilde{\mathcal C}_1$.

By the same argument with the proof of Lemma $4.3$ in \cite{Wang-Zhou14},
we see that $u_0$ is a generalized solution to the degenerate Monge-Amp\`ere
equation
$$
\det D^2u=0.
$$
By convexity, $\mathcal T=\{x\in P\,|\, u_0(x)=0\}$ is convex. 
Moreover
any extreme point of $\mathcal T$ is a boundary point of $P$
by Lemma $4.1$ in \cite{Wang-Zhou11}.
Since $P$ is two dimensional, $\mathcal T$ is either a line segment through
$p_0$ with both endpoints on $\partial P$ or a convex polygon with vertices on
$\partial P$. Note here that if the dimension of $P$ is greater than two
the convex set $\mathcal T$ may be more complicated.
We set an affine function $L$ on $\mathbf{R}^2$ as follows.
When $\mathcal T$ is a line segment,
$$
L(x):=\langle n,x-p_0\rangle,
$$
where $n$ is a unit normal vector of $\mathcal T$.
When $\mathcal T$ is a polygon, 
$$
L(x):=\langle n,x-p_1\rangle,
$$
where $p_1\in \partial \mathcal T\setminus \partial P$ and
$n$ is the outer unit normal vector of $\partial \mathcal T$
at $p_1$. In either case,
$\psi=\max\{0,L\}$ is a sPL convex function with nonempty crease.

We next define a function $a$ as
$$
a(p)=\lim_{t\searrow 0}\frac{u_0(p+tn)-u_0(p)}{t}.
$$
Here
$p\in \mathcal T$ when $\mathcal T$ is a line segment or
$p$ is in the edge of $\mathcal T$ containing $p_1$ when
$\mathcal T$ is a polygon. By convexity of $u_0$, the limit exists and is
nonnegative for any $p$.
\begin{lem}
$a_0:=\inf a=0$.
\end{lem}
\begin{proof}We give a proof only when $\mathcal T$ is a line segment since
the case when $\mathcal T$ is a polygon is similar.
Suppose $a_0>0$. Denote $u':=u_0-a_0\psi$.
Then
$\displaystyle{\int_{\partial P}\alpha u'\,d\sigma}<1$.
By the definition of $a_0$, $u'$ is convex on $P^*$ and
$$
\inf _Pu'=u'(p_0)=u_0(p_0)-a_0\psi(p_0)=0.
$$
Since $\mathcal L(\psi)>0$ by assumption,
$$
\mathcal L(u_0)=\mathcal L(u')+a_0\mathcal L(\psi)>\mathcal L(u').
$$
Hence, since $\Tilde{u'}:=
\displaystyle{\left(\int_{\partial P}\alpha u'\,d\sigma\right)^{-1}}u'
\in \Tilde{\mathcal C}_1$,
$$
0\ge \mathcal L(u_0)> \mathcal L(u')>\mathcal L(\Tilde{u'}).
$$
This is a contradiction.
\end{proof}
By the definition of $\mathcal T$ and $L$,
$u_0$ is positive on $P\cap \{L>0\}$. For any $\varepsilon >0$,
$G_\varepsilon:=\{x\in P\,|\,u_0(x)<\varepsilon \psi(x)\}$
is nonempty because $a_0=0$.
Since $\mathcal T\subset \{L\le 0\}$, there exists $\delta(\varepsilon)>0$
such that
$G_\varepsilon \subset \{0\le L<\delta(\varepsilon)\}$ and
$\lim_{\varepsilon \searrow 0}\delta(\varepsilon)=0$.
Denote
$$
u_1:=u_0\chi_{-},\ \ u_2:=(u_0-\varepsilon \psi)\chi_+,\ \ 
\Tilde{u_2}:=\max \{0,u_2\},
$$
where
$$
\chi_-(x)=
\begin{cases}
1 & \text{when }L(x)<0\\
0 & \text{otherwise}
\end{cases},\ \ 
\chi_+=1-\chi_-.
$$
It is easy to see that
$u_1+\Tilde{u_2}\ge 0$ is convex and
$(u_1+\Tilde{u_2})(p_0)=0$.
Denote $\Tilde{u}:=u_1+\Tilde{u_2}+\varepsilon \psi$.
Then we have
$$
\Tilde{u}-u_0=
\Tilde{u_2}-u_2
=
\begin{cases}
-u_2=\varepsilon L-u_0\le \varepsilon \delta(\varepsilon) & \text{on }G_\varepsilon,\\
0 & \text{on }G_\varepsilon^c.
\end{cases}
$$
Hence there exsits a positive constant $C$ such that
$$
\mathcal L(\Tilde{u}-u_0)=
\int_{\partial P}\alpha(\Tilde{u}-u_0)\,d\sigma-
\int_PA\beta(\Tilde{u}-u_0)\,d\mu <C\varepsilon \delta(\varepsilon).
$$
Therefore we have
$$
\mathcal L(u_1+\Tilde{u_2})=\mathcal L(\Tilde{u})-\varepsilon \mathcal L(\psi)
<\mathcal L(u_0)+\varepsilon(C\delta(\varepsilon)-\mathcal L(\psi))<\mathcal L
(u_0).
$$
for any sufficiently small $\varepsilon>0$.
Denote $u_3:=\displaystyle{\left(\int_{\partial P}\alpha(u_1+\Tilde{u_2})\,
d\sigma\right)^{-1}(u_1+\Tilde{u_2})}\in \Tilde{\mathcal C}_1$.
Since $u_1+\Tilde{u_2} \le u_0$,
$\displaystyle{\int_{\partial P}\alpha(u_1+\Tilde{u_2})\,d\sigma \le 1}$.
Therefore we obtain
$\mathcal L(u_3)\le \mathcal L(u_1+\Tilde{u_2})<\mathcal L(u_0)$.
This is a contradiction. This completes the proof of Theorem \ref{gen-D,W-Z}.
\end{proof}

Finally we observe by numerical analysis that there exists a Killing vector field which 
gives a toric K-stable pair in the sense of Apostolov-Maschler.

Let $\Delta_p$ be the convex hull of
$(0,0),(p,0),(p,1-p)$ and $(0,1)$
for $0<p<1$. By Delzant construction, 
the K\"ahler class of a toric K\"ahler metric on the 
one point blow up of $\mathbf{C}P^2$
corresponds to $\Delta_p$ up to multiplication of a positive constant.

Denote
$$
\mathcal P:=\{(a,b,c)\in \mathbf{R}^3\,|\, 
c>0, \,ap+c>0, \,ap+b(1-p)+c>0, \,b+c>0\}.
$$
An affine function $a\mu_1+b\mu_2+c$ is positive on $\Delta_p$ if and only if
$(a,b,c)\in \mathcal P$. By the argument in Section $3$ and $4$,
$\mathrm{Fut}_{\Delta_p,a\mu_1+b\mu_2+c,-2,n}$ vanishes if and only if
$(a,b,c)\in \mathcal P$ is a critical point
of
$$
EH_n(a,b,c):=\dfrac{\displaystyle{\int_{\partial \Delta_p}(a\mu_1+b\mu_2+c)^{2-n}
\,d\sigma}}
{\displaystyle{\left(\int_{\Delta_p}(a\mu_1+b\mu_2+c)^{-n}\,d\mu\right)^{\frac{n-2}{n}}}}.
$$
For $n=4$, the authors identified in \cite{FO17} such critical points as follows:
\begin{equation*}
\begin{split}
&(a)\ \ C\left(1,0,\dfrac{p(1-\sqrt{1-p})}{2\sqrt{1-p}+p-2}\right),\ C>0,\ 0<p<1,\\
&(b)\ \ C\left(-1,0,\dfrac{p(3p\pm \sqrt{9p^2-8p})}{2(p\pm \sqrt{9p^2-8p})}\right),\ \ 
C>0,\ \frac89 <p<1,\\
&(c)\ \ C\left(-p^2+4p-2\pm \sqrt{F(p)},\pm2 \sqrt{F(p)},-p^2-2p+2\mp \sqrt{F(p)}
\right),\ \ 
C>0,\ 0<p<\alpha,
\end{split}
\end{equation*}
where $\alpha\approx 0.386$ is a real root of
$$
F(x):=x^4-4x^3+16x^2-16x+4=0.
$$
For the affine functions corresponding to $(a)$ and $(b)$,
LeBrun gave concrete examples of cKEM metrics in \cite{L2}.
Hence $(\Delta_p,a\mu_1+b\mu_2+c,4)$ is K-polystable by Corollary $3$
in \cite{AM}. On the other hand, in case $(c)$, 
we do not know whether 
there exists cKEM metrics.
Denote 
\begin{equation*}
\begin{split}
f^\pm_p&=
(-p^2+4p-2\pm \sqrt{F(p)})\mu_1\pm2 \sqrt{F(p)}\mu_2-
p^2-2p+2\mp \sqrt{F(p)}\\
&=:a_p^\pm \mu_1+b_p^\pm \mu_2+c_p^\pm.
\end{split}
\end{equation*}
By Theorem \ref{gen-D,W-Z}, if $\mathrm{DF}_{\Delta_p,f^\pm_p,4}(\phi)$ is positive
for any sPL convex function $\phi$,
$(\Delta_p,f^\pm_p,4)$ is K-polystable.
According to the position of the boundary points ${\bf u,v}$
of creases, we divide into the following six cases.

\begin{description}
\item[$1.\ {\bf u}=(0,e),\ {\bf v}=(p,f)\ (0\le e\le 1,0\le f\le 1-p)$ ]

In this case, the corresponding sPL convex function is
$\phi=\max\{(f-e)\mu_1-p\mu_2+pe,0\}$.
Then
\begin{align*}
\int_{\partial \Delta_p}\frac{\phi }{(f_p^\pm)^3}\,d\sigma
=&\int_0^p\frac{(f-e)\mu_1+pe}{(a_p^\pm \mu_1+c_p^\pm)^3}\,d\mu_1+
\int_0^f \frac{p(f-\mu_2)}{(a_p^\pm p+b_p^\pm \mu_2+c_p^\pm)^3}\,d\mu_2\\
&+\int_0^e\frac{p(e-\mu_2)}{(b_p^\pm \mu_2+c_p^\pm)^3}\,d\mu_2
\end{align*}
and
$$
\int_{\Delta_p}\frac{\phi}{(f_p^\pm)^5}\,d\mu
=\int_0^p\,d\mu_1\int_0^{\frac{f-e}{p}\mu_1+e}
\frac{(f-e)\mu_1-p\mu_2+pe}{(f_p^\pm)^5}\, d\mu_2
$$

It is too long and complicated to give the full description of $\mathrm{DF}_{\Delta_p,
f_p^\pm,4}(\phi)$. We put the graph of $\mathrm{DF}_{\Delta_{0.1},f^-_{0.1},4}$,
as a function of $(e,f)$,
instead.
All graphics in this article are drawn by \textit{Mathematica}.

\includegraphics{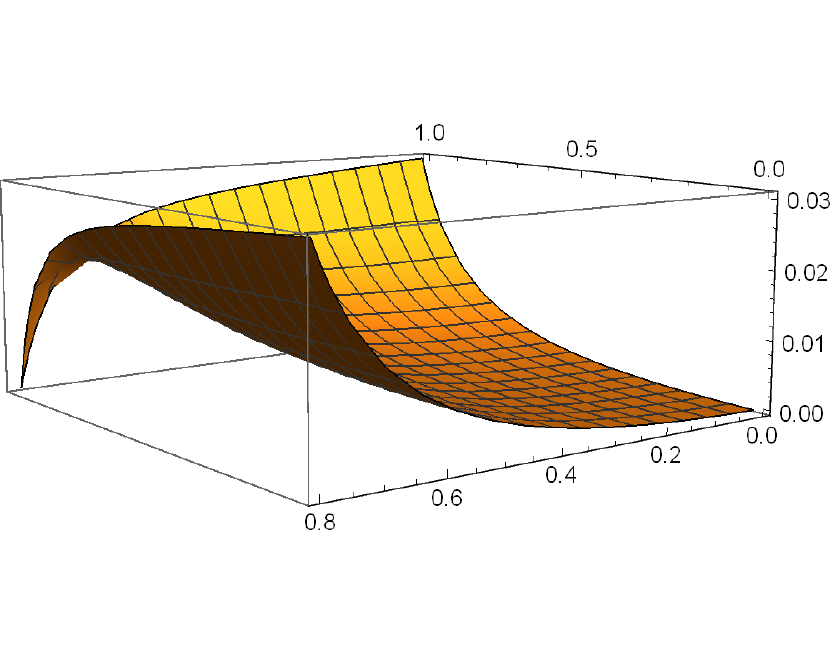}

\item[$2.\ {\bf u}=(e,0),\ {\bf v}=(f,1-f)\ (0\le e\le p,0\le f\le p)$ ]
In this case, the corresponding sPL convex function is
$\phi=\max\{(f-1)\mu_1+(f-e)\mu_2+(1-f)e,0\}$.
Then
\begin{align*}
\int_{\partial \Delta_p}\frac{\phi }{(f_p^\pm)^3}\,d\sigma
=&\int_0^e\frac{(1-f)(e-\mu_1)}{(a_p^\pm \mu_1+c_p^\pm)^3}\,d\mu_1+
\int_0^1 \frac{(f-e)\mu_2+(1-f)e}{(b_p^\pm \mu_2+c_p^\pm)^3}\,d\mu_2\\
&+\int_0^f\frac{(f-1)\mu_1+(f-e)(1-\mu_1)+(1-f)e}{(a_p^\pm \mu_1+
b_p^\pm (1-\mu_1)+c_p^\pm)^3}\,d\mu_1
\end{align*}
and
\begin{align*}
\int_{\Delta_p}\frac{\phi}{(f_p^\pm)^5}\,d\mu
=&\int_0^{1-f}\,d\mu_2\int_0^{\frac{f-e}{1-f}\mu_2+e}
\frac{(f-1)\mu_1+(f-e)\mu_2+(1-f)e}{(f_p^\pm)^5}\, d\mu_1\\
&+\int_{1-f}^1\,d\mu_2\int_0^{1-\mu_2}
\frac{(f-1)\mu_1+(f-e)\mu_2+(1-f)e}{(f_p^\pm)^5}\, d\mu_1
\end{align*}
The graph of $\mathrm{DF}_{\Delta_{0.1},f^-_{0.1},4}$ is as follows.

\includegraphics{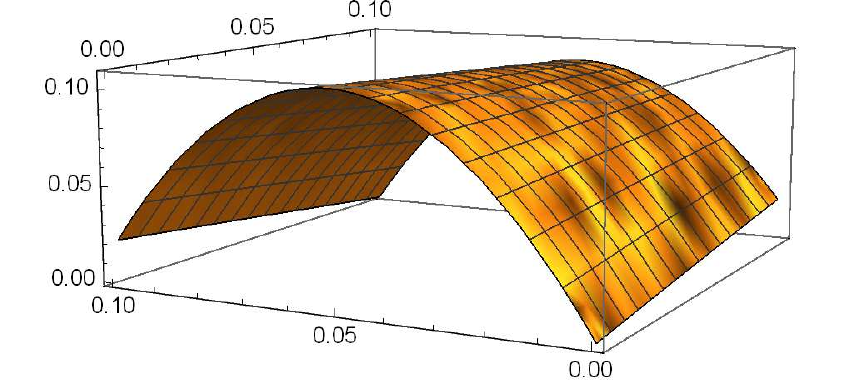}

\item[$3.\ {\bf u}=(0,e),\ {\bf v}=(f,1-f)\ (0\le e\le 1,0\le f\le p)$ ]
In this case, the corresponding sPL convex function is
$\phi=\max\{(f+e-1)\mu_1+f\mu_2-fe,0\}$.
Then
\begin{align*}
\int_{\partial \Delta_p}\frac{\phi }{(f_p^\pm)^3}\,d\sigma
=&\int_e^1\frac{f(\mu_2-e)}{(b_p^\pm \mu_2+c_p^\pm)^3}\,d\mu_2\\
&+
\int_0^f \frac{(f+e-1)\mu_1+f(1-\mu_1)-fe}{(a_p^\pm \mu_1+b_p^\pm (1-\mu_1)+
c_p^\pm)^3}\,d\mu_2
\end{align*}
and
$$
\int_{\Delta_p}\frac{\phi}{(f_p^\pm)^5}\,d\mu
=\int_0^f\,d\mu_1\int_{\frac{1-f-e}{f}\mu_1+e}^{1-\mu_1}
\frac{(f+e-1)\mu_1+f\mu_2-fe}{(f_p^\pm)^5}\, d\mu_2
$$
The graph of $\mathrm{DF}_{\Delta_{0.1},f^-_{0.1},4}$ is as follows.

\includegraphics{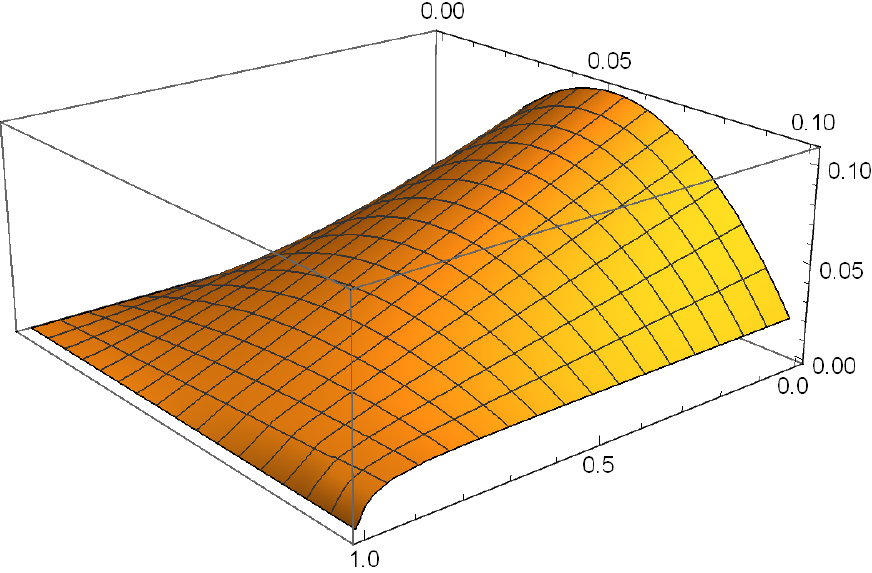}

\item[$4.\ {\bf u}=(0,e),\ {\bf v}=(f,0)\ (0\le e\le 1,0\le f\le p)$ ]
In this case, the corresponding sPL convex function is
$\phi=\max\{-e\mu_1-f\mu_2+fe,0\}$.
Then
$$
\int_{\partial \Delta_p}\frac{\phi }{(f_p^\pm)^3}\,d\sigma
=\int_0^f\frac{e(f-\mu_1)}{(a_p^\pm \mu_1+c_p^\pm)^3}\,d\mu_1+
\int_0^e \frac{f(e-\mu_2)}{(b_p^\pm \mu_2+c_p^\pm)^3}\,d\mu_2
$$
and
$$
\int_{\Delta_p}\frac{\phi}{(f_p^\pm)^5}\,d\mu
=\int_0^f\,d\mu_1\int_0^{-\frac{e}{f}\mu_1+e}
\frac{-e\mu_1-f\mu_2+fe}{(f_p^\pm)^5}\, d\mu_2
$$
The graph of $\mathrm{DF}_{\Delta_{0.1},f^-_{0.1},4}$ is as follows.

\includegraphics{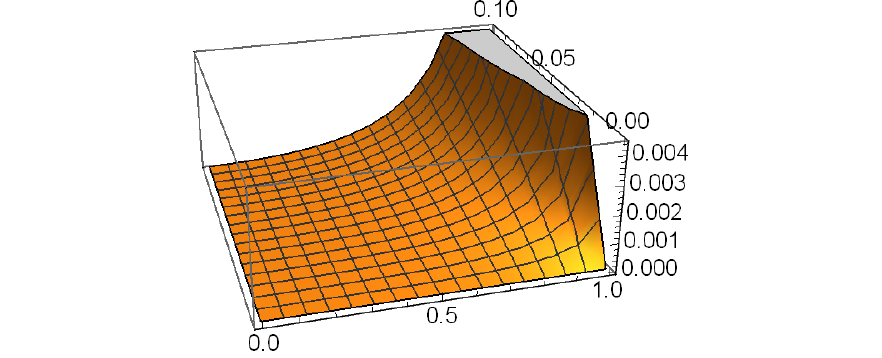}

\item[$5.\ {\bf u}=(p,e),\ {\bf v}=(f,0)\ (0\le e\le 1-p,0\le f\le p)$ ]
In this case, the corresponding sPL convex function is
$\phi=\max\{e\mu_1+(f-p)\mu_2-fe,0\}$.
Then
$$
\int_{\partial \Delta_p}\frac{\phi }{(f_p^\pm)^3}\,d\sigma
=\int_f^p\frac{e(\mu_1-f)}{(a_p^\pm \mu_1+c_p^\pm)^3}\,d\mu_1+
\int_0^e \frac{(p-f)(e-\mu_2)}{(a_p^\pm p+b_p^\pm \mu_2+c_p^\pm)^3}\,d\mu_2
$$
and
$$
\int_{\Delta_p}\frac{\phi}{(f_p^\pm)^5}\,d\mu
=\int_f^p\,d\mu_1\int_0^{\frac{e}{p-f}(\mu_1-p)+e}
\frac{e\mu_1+(f-p)\mu_2-fe}{(f_p^\pm)^5}\, d\mu_2
$$
The graph of $\mathrm{DF}_{\Delta_{0.1},f^-_{0.1},4}$ is as follows.

\includegraphics{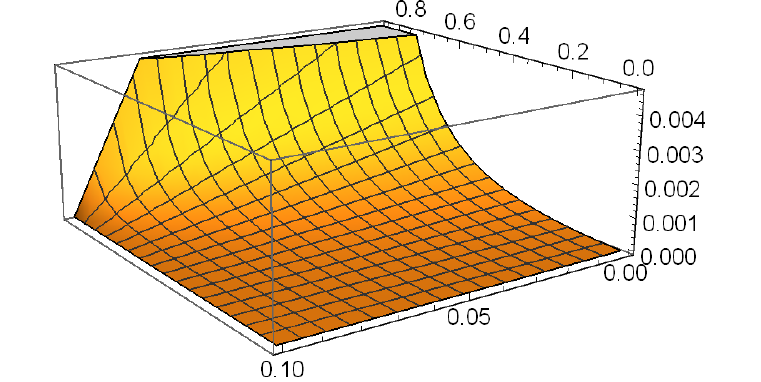}

\item[$6.\ {\bf u}=(p,e),\ {\bf v}=(f,1-f)\ (0\le e\le 1-p,0\le f\le p)$ ]
In this case, the corresponding sPL convex function is
$\phi=\max\{(1-e-f)(\mu_1-p)+(p-f)\mu_2+(f-p)e,0\}$.
Then
\begin{align*}
\int_{\partial \Delta_p}\frac{\phi }{(f_p^\pm)^3}\,d\sigma
=&\int_f^p\frac{(1-e-f)(\mu_1-p)+(p-f)(1-\mu_1)+(f-p)e}{(a_p^\pm \mu_1+
b_p^\pm (1-\mu_1)+c_p^\pm)^3}\,d\mu_1\\
&+\int_e^{1-p}\frac{(p-f)(\mu_2-e)}{(a_p^\pm p+b_p^\pm \mu_2+c_p^\pm)^3}\,d\mu_2
\end{align*}
and
$$
\int_{\Delta_p}\frac{\phi}{(f_p^\pm)^5}\,d\mu
=\int_f^p\,d\mu_1\int_{\frac{e+f-1}{p-f}(\mu_1-p)+e}^{1-\mu_1}
\frac{(1-e-f)(\mu_1-p)+(p-f)\mu_2+(f-p)e}{(f_p^\pm)^5}\, d\mu_2
$$
The graph of $\mathrm{DF}_{\Delta_{0.1},f^-_{0.1},4}$ is as follows.

\includegraphics{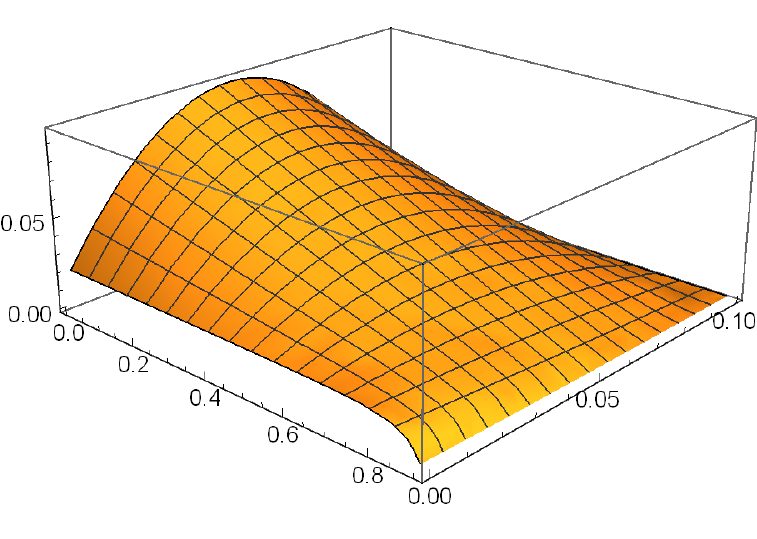}
\end{description}

Looking at the graphs, 
$(\Delta_p,f^\pm_p,4)$ must be K-polystable.
By Theorem $5$ in \cite{AM},
cKEM metrics with Killing potential $f_p^\pm$
ought to exist.
We leave this problem to the interested readers.




\end{document}